\documentclass[oneside,english]{amsart}
\usepackage[T1]{fontenc}
\usepackage[latin9]{inputenc}
\usepackage{textcomp}
\usepackage{amsthm}
\usepackage{amstext}
\usepackage{amssymb}
\usepackage[all]{xy}
\usepackage{url}
\usepackage{color}   
\usepackage{hyperref}

\makeatletter

\newcommand{\lyxmathsym}[1]{\ifmmode\begingroup\def\b@ld{bold}
  \text{\ifx\math@version\b@ld\bfseries\fi#1}\endgroup\else#1\fi}

\numberwithin{equation}{section}
\numberwithin{figure}{section}
\theoremstyle{plain}
\newtheorem{thm}{Theorem}
  \theoremstyle{remark}
  \newtheorem{rem}[thm]{Remark}
  \theoremstyle{plain}
\newtheorem{example}[thm]{Example}
  \theoremstyle{plain}
  \newtheorem{prop}[thm]{Proposition}
  \theoremstyle{plain}
  \newtheorem{lem}[thm]{Lemma}
  \theoremstyle{definition}
  \newtheorem{defn}[thm]{Definition}
  \theoremstyle{remark}
  \newtheorem*{rem*}{Remark}
  \theoremstyle{plain}
  \newtheorem{cor}[thm]{Corollary}

\makeatother

\usepackage{babel}

\begin{document}

\title[A geometric description of cohomology for Hilbert manifolds]{Hilbert stratifolds and a Quillen type geometric description of cohomology for Hilbert manifolds}

\author{Matthias Kreck and Haggai Tene}

\address{Max-Planck Institut \& Mathematisches Institut, Rheinische Friedrich-Wilhelms-Universit\"at, Bonn, Germany} 
\email{kreck@math.uni-bonn.de}

\address{Mathematisches Institut, Universit\"at Heidelberg} 
\email{tene@mathi.uni-heidelberg.de}

\begin{abstract} In this paper we use tools from differential topology to give a geometric description of cohomology for Hilbert manifolds. Our model is Quillen's geometric description of cobordism groups for finite dimensional smooth manifolds \cite{Q}. Quillen stresses the fact that this construction allows the definition of Gysin maps for ''oriented'' proper maps. For finite dimensional manifolds one has a Gysin map in singular cohomology which is based on Poincar\'e duality, hence it is not clear how to extend it to infinite dimensional manifolds.
But perhaps one can overcome this difficulty by giving a Quillen type description of singular cohomology for Hilbert manifolds. This is what we do in this paper. Besides constructing a general Gysin map, one of our motivations was a geometric construction of equivariant cohomology, which even for a point is the cohomology of the infinite dimensional space $BG$, which has a Hilbert manifold model. Besides that, we demonstrate the use of such a geometric description of cohomology by several other applications. 
We give a quick description of characteristic classes of a finite dimensional vector bundle and apply it to a generalized Steenrod representation problem for Hilbert manifolds and define a notion of a degree of proper oriented Fredholm maps of index $0$.

\end{abstract}
\maketitle
\setcounter{tocdepth}{2}
\tableofcontents

\section{Introduction} Quillen \cite{Q} has given a geometric description of cobordism groups for finite dimensional smooth manifolds $M$. Those were originally defined in terms of the Thom spectrum.  His description is, roughly, in terms of bordism classes of {\bf proper} maps from manifolds to $M$. This looks almost identical to the description of the corresponding homology groups which are given as bordism classes of maps from {\bf compact} manifolds to $M$. Quillen \cite {Q}, page 30, gives a motivation for his work: "I  have  been  strongly influenced by Grothendieck's  theory of motives in algebraic geometry and like to think of a cobordism 
theory as a universal contravariant functor on the category of $C^\infty$ manifolds endowed with a Gysin homomorphism for a class of proper ''oriented'' maps, instead of as the generalized cohomology theory given by a specific Thom spectrum." In singular cohomology a Gysin map for "oriented" proper maps between finite dimensional manifolds can be defined using Poincar\'e duality to locally finite homology.  Can one extend this to infinite dimensional manifolds? 

Even before Quillen, the idea of a geometric description of special cohomology classes in a Hilbert manifold $X$ came up in a paper by Eells (\cite {Ee}, page 801 (D)): "Alexander-Pontrjagin duality suggests that a p-codimensional  submanifold $A$ of $X$ should represent some sort of homology class of $X$ of dimension $\infty - p$ ... We are now not in the possession of a definite theorem of that kind."

These two statements are our main motivation to look for a geometric description of singular cohomology for Hilbert manifolds. The idea is to give a Quillen type construction of singular cohomology for Hilbert manifolds. For finite dimensional manifolds the first author \cite{K1} has given a geometric description of singular cohomology using bordism classes of certain stratified spaces, called stratifolds, instead of manifolds. Our idea is to generalize this to infinite dimensions.

The obvious thing to do would be to define Hilbert stratifolds generalizing the finite dimensional stratifolds. We don't see a straightforward way to do this. For the description of cohomology groups of finite dimensional manifolds we do not need to define stratifolds on their own.  The cocycles of our theory are singular stratifolds, these are proper maps $f: \mathcal S \to M$ from a stratifold to $M$. The concept of singular stratifolds in $M$ can be generalized to {\bf oriented singular Hilbert stratifolds $f: \mathcal S \to M$}, where $M$ is a Hilbert manifold. Then one can consider cobordism groups $SH^k(M)$ as in the finite dimensional setting and obtain a contravariant functor. As in the finite dimensional  setting Gysin maps are only defined for certain maps $f: M \to N$, in our situation these are {\bf proper oriented Fredholm maps}. We also define a non oriented version, which leads to cohomology groups $SH^k(M,\mathbb Z /2)$. We will explain the details later. Our main theorems are:

\begin{thm} \label{thm1}The functors $SH^k(M)$ on Hilbert manifolds $M$ together with the coboundary operator $d$ form a cohomology theory on Hilbert manifolds, by which we mean that those are homotopy functors, the coboundary operator $d$ is natural and the Mayer-Vietoris sequence is exact. \\
If $f: M \to N$ is a proper oriented Fredholm map of index $r$, then composition with $f$ defines a Gysin map
$$
f_!: SH^k(M) \to SH^{k-r}(N).
$$
The analogous statements hold for $SH^k(M;\mathbb Z/2)$. 
\end{thm}
\begin{thm} \label{natural iso} There are natural isomorphisms of multiplicative cohomology theories
$$
H^k(M;\mathbb Z) \to SH^k(M),
$$
and
$$
H^k(M;\mathbb Z/2) \to SH^k(M;\mathbb Z/2).$$

\end{thm}

There are other reasons for looking for a geometric theory. The geometric nature of singular homology (one catches a hole by a chain of simplices) is very useful both in theory and in computations. One would like to use similar methods for the more "abstract" singular cohomology. For example, representing cohomology classes by submanifolds can be useful for computing cup products by intersection or for computing induced maps. Also characteristic classes have a simple description in this language. 
Since not all cohomology classes are represented by submanifolds, or even by proper maps from a manifold, one needs, in general, singular objects like stratified spaces. Given this one can compute cup products and similar pairings geometrically by taking transversal intersections.

Another motivation for  our work is an extension of the geometric description of singular cohomology to equivariant cohomology (in terms of the Borel construction) for compact Lie group $G$ actions. In the case of equivariant homology such an extension is very simple. The groups are given by bordism classes of equivariant maps from compact stratifolds with a free $G$-action to $M$ \cite{T2}. Similar groups were constructed by MacPherson \cite{Mac} using pseudomanifolds instead of stratifolds. As in the non-equivariant case, one could consider the corresponding theory given by proper equivariant maps from stratifolds with a free $G$-action to $M$. This is a cohomology theory (which was studied in \cite{T2}), but it is not the cohomology of the Borel construction. The modification needed is to replace $M$ by $M \times EG$, as expected from the Borel construction. Since $EG$ is not a finite dimensional manifold (unless $G$ is trivial) we can no longer work in the finite dimensional setting. To solve this, we note that there is a model for $EG$ as a Hilbert manifold. If one considers $M \times EG$, it is clear that also the stratifolds have to be replaced by stratifolds whose strata are Hilbert manifolds.

Besides the geometric description of equivariant cohomology with the main application of the construction of a general Gysin map, we will demonstrate the use of a geometric description in some examples. We give a quick description of characteristic classes of finite dimensional vector bundles and define a notion of degree for proper Fredholm maps of index 0. 

\section{Hilbert Manifolds and Fredholm Maps}

In this part we discuss some basic properties of Hilbert manifolds
and Fredholm maps. We refer to \cite{L} regarding Hilbert manifolds
and to the appendix in \cite{A} regarding Fredholm maps. 

\subsection{Definitions and fundamental properties}

~

A smooth Hilbert manifold is a manifold modeled on a real Hilbert
space with smooth transition maps. We will always assume that the
Hilbert spaces $H$ are {\bf infinite dimensional and separable and the manifolds
are Hausdorff and admit a countable basis}. From now on we will refer
to smooth Hilbert manifolds just as Hilbert manifolds. A pair $(M,\partial M)$
is called a Hilbert manifold with boundary if it has an atlas modelled
on the upper half space - $[0,\infty)\times H$. The boundary
$\partial M$ is the set of points which are mapped onto $H\times\{0\}$
by some chart. It can be shown that the boundary is a Hilbert
manifold and has a collar. 

The theory of Hilbert manifolds enjoys many nice properties, we state a few:
\begin{enumerate}
\item (\cite{EE}, Theorem 1A) Every infinite dimensional smooth Hilbert
manifold is diffeomorphic to an open subset of the Hilbert
space.
\item (\cite{Ku}, Theorem 2) Let $H$ be an infinite dimensional (real
resp.\ complex) Hilbert space then the general linear group $GL(H)$
(over $\mathbb{R}$ resp.\ $\mathbb{C}$) is contractible. This implies
that every vector bundle with structure group $GL(H)$ over a space which has the homotopy type of a $CW$ complex is
trivial.
\item (\cite{BH}, Theorem 0.1) Any homotopy equivalence of pairs $f:(N,\partial N)\to(M,\partial M)$
between smooth infinite dimensional Hilbert manifolds with boundary
is homotopic to a diffeomorphism of pairs.
\end{enumerate}
Let $V$ and $W$ be two Hilbert spaces. A bounded linear operator
$f:V\to W$ having both finite dimensional kernel and cokernel is
called a Fredholm operator. The index of $f$ is defined to be: \[
index(f)=dim(ker \, \, f)-dim(coker \, \, f).\]
Note that if $V$ and $W$ are finite dimensional then the index is independent of $f$: \[
index(f)=dim(V)-dim(W).\]

This definition can be extended to smooth maps between Hilbert manifolds.
Let $M$ and $N$ be two Hilbert manifolds. A smooth map $f:M\to N$
is called a {\bf Fredholm map} if the differential at each point is Fredholm.
The index of $f$ at a point $x\in M$ is given by $ind(df_{x})$,
and it is locally constant. 
If $\left\{ ind(df_{x})|x\in M\right\} $ is bounded from above then
we call the upper bound the index of $f$ and denote it by $ind(f)$.

\section{Singular Hilbert Stratifolds}

We now define Hilbert stratifolds. We first repeat the definition of finite dimensional stratifolds from \cite{K1}  since the basic structure is the same, although the setting is diffferent. We consider a toplogical space $\mathcal S$ togehter with its sheaf of continuous fuctions, considered as a sheaf of $\mathbb R$-algebras. We will consider subsheaves $F$ of the sheaf of continuous functions, where $F(U)$ is a subalgebra for each open subset $U$. Without extra mentioning, all sheaves in the following are of this type. This is the only datum we need for the definition of stratifolds, the stratifolds are given by such a sheaf $F$ which fulfills certain properties. 

To formulate the properties we introduce certain notations. If $(\mathcal S,F)$ and $(\mathcal S',F')$ are as before we call a continuous map $f: \mathcal S \to \mathcal S'$ a  {\bf smooth map from $(\mathcal S,F)$ to $(\mathcal S',F')$} if $\rho f \in F(f^{-1}(U)$ for all $\rho \in F'(U)$.  A bijective map $f$, such that $f$ and $f^{-1}$ are smooth, is called a {\bf diffeomorphism}. If  $M$ is a smooth manifold and $F$ is the sheaf of smooth functions then this definition of smooth functions and diffeomorphisms agrees with the ordinary definition in terms of local coordinats.  In turn, a pair $(\mathcal S,F)$ which is locally diffeomorophic to $(\mathbb R^n, C^{\infty}( \mathbb R^n))$ induces a natural structure of a smooth manifold on $\mathcal S$ where the smooth atlas is given by all local diffeomorphisms to $(\mathbb R^n, C^{\infty} (\mathbb R^n))$.

There are seven conditions a stratifold has to fulfill. In particular, these conditions will imply that $\mathcal S$ is the union of  pairwise disjoint subsets called strata, such that these subsets, together with the restriction of the sheaf $F$, are smooth manifolds. To define these subsets we recall that given the sheaf $F$  we can consider at a point $x \in S$ the stalk $\Gamma(F)_x$ of smooth functions and define the {\bf tangent space $T_x(\mathcal S)$  at $x$}  as the vector space of derivations of $\Gamma (F)_x$. If $(\mathcal S,F)$ is a smooth manifold this is one of the definitions  of the  tangent space. We define the {\bf $r$-stratum} as 
$$
\mathcal S^r := \{x \in \mathcal S|\, dim T_x(\mathcal S) = r\},
$$
and the {\bf $r$-skeleton} as 
$$
\Sigma ^r := \cup _{i\le r} \mathcal S^i.
$$

\begin{defn} An {\bf $n$-dimensional stratifold} is a pair $(\mathcal S,F)$ such that the following conditions are fulfilled:\\
I) $\mathcal S$ is a Hausdorff space admitting a countable basis. \\
II) $\mathcal S^r$ is empty for  $r>n$.\\
III) The skeleta are closed subspaces. \\
IV) For each $x \in \mathcal S$ and open neighborhood $U$ there is a function $\rho: \mathcal S \to \mathbb R_{\ge 0}$ in $F$ such that $supp \, \rho \subseteq U$ and $\rho (x) \ne 0$.\\ 
V) The strata together with the restriction of $F$ are smooth manifolds.\\
VI) For each $x \in \mathcal S^r$ the restriction of germs is an isomorphism:
$$\Gamma (F)_x \to \Gamma (F|_{\mathcal S^r})_x.$$\\
VII)  If $f_1$, $f_2$, ..., $f_n$ are smooth functions in $F$ and $g: \mathbb R^n \to \mathbb R$ is smooth, then $g(f_1,...,f_n)$ is a smooth function in $F$.\\

\end{defn}

\begin{rem}
Condition IV implies that $\mathcal S$ is a regular topological space, which together with the fact that it is second countable (by I), implies it is paracompact. It is easy to see that using paracompactness and the functions in IV, one can obtain a smooth partition of unity with respect to every open covering.
\end{rem}

For more details about  stratifolds see \cite{K1}. In this book bordism classes of certain stratifolds were used to define a homology theory for topological spaces, which for $CW$-complexes agrees with singular homology, and a cohomology theory for smooth manifolds, which agrees with singular cohomology. In this paper we study the generalization of the cohomology theory to infinite dimensional manifolds. 

It is natural to ask how this relates to other concepts of stratified spaces in differential topology, in particular Mather's abstract pre-stratified spaces \cite{Ma}. Anna Grinberg has analyzed this carefully and proved that an abstract pre-stratified space has the structure of a stratifold. But Mather has additional data, which we don't need for our purpose, namely a certain control function $\rho$. He also uses the language of tubular neighborhoods and retracts, whereas we use the sheaf language in which we formulate conditions which give Mather's neighborhoods and retracts. For details we refer to Grinberg's paper \cite{Gr}. 

There is an obvious problem with a generalization of the concept of stratifolds to infinite dimensions, since the dimension of the tangent spaces cannot be used to distinguish the strata. For that reason, we have to pass to singular Hilbert stratifolds in a Hilbert manifold $M$. For our purpose this is enough since the singular Hilbert stratifolds in $M$ will be the cocycles of our cohomology theory - as in the finite dimensional setting. 

Let $M$ be a Hilbert manifold which we equip with the sheaf of smooth functions. We consider a topological space $\mathcal S$ together with a subsheaf $F$ of the sheaf of continuous functions as before togehter with a smooth map $f: \mathcal S \to M$. Now we can consider for each $x \in \mathcal S$ the differential $df_x: T_x \mathcal S \to T_xM$ of $f$ at $x$. We call $f$ a {\bf Fredholm map} if $df_x$ has finite kernel and cokernel for each $x \in \mathcal S$. At each point $x \in \mathcal S$ we consider the index $ind \, df_x:= dim \, ker (df_x) - dim \, coker (df_x)$. We note that if $M$ is a point, then if $f : \mathcal S \to pt$ is  a Fredhom map, this implies that all strata as defined above are finite dimensional and the dimension is the index of $df_x$ at a point $x$  in the stratum. Thus the following is a direct generalization of the concept of strata to the infinite dimensional setting, if we consider the case where $M$ is a point. 

We define the {\bf $r$-stratum} as 
$$
\mathcal S^r := \{x \in \mathcal S|\, ind \, df_x = r\},
$$
and the {\bf $r$-skeleton} as 
$$
\Sigma ^r := \cup _{i\le r} \mathcal S^i.
$$

With these concepts the definition of a singular Hilbert stratifold is a generailzation of the finite dimensional setting:

\begin{defn} A {\bf singular Hilbert $n$-stratifold} in a Hilbert manifold $M$ is given by a space $\mathcal S$ together with a sheaf $F$ of continuous functions as before and a proper Fredholm map 
$$
f : \mathcal S \to M
$$
such that the conditions I) - VII) are fulfilled. Two singular Hilbert $n$-stratifolds $(\mathcal S ,f)$ and $(\mathcal S' ,f')$ are isomorphic if there is an homeomorphism $g:\mathcal S \to \mathcal S'$ which induces an isomorphism of sheaves and such that $f'g=f$.

For the definition of cohomology classes we need - as in the finite dimensional setting - the concept of {\bf regular singular Hilbert stratifolds}. This means that each $x$ in the $k^{th}$ stratum $\mathcal S^k$ has an open neighbourhood $V$ in $\mathcal S$ and a finite dimensional stratifold $F$ whose $0^{th}$ stratum consists of a single point $*$ such that the restriction to $V$ is isomorphic to $(V \cap \mathcal S^k)  \times F$, and the restriction of $f$ to $(V \cap \mathcal S^k)\times \{ *\}$ is given by the projection on the first factor.   

\end{defn}

The constructions made in \cite {K1}  generalize without any change to the infinite dimensional setting. In particular the construction of the product of two singular Hilbert stratifolds using local retractions works in our context. This was used above  when we defined regular stratifolds and will later be used for  the cross product in cohomology. 

\begin{rem}
For finite dimensional stratifolds the number of non-empty strata is automatically finite. In the infinite dimensional setting this is not implied and we don't require it. The reason is that otherwise the theory we get is, a priory, not an additive cohomology theory in the sense of Milnor, i.e. the cohomology of a disjoint union is the product of the cohomology of the components.
\end{rem}

Simple examples of singular Hilbert stratifolds are given by the product of a finite dimensional singular stratifold $f : \mathcal S \to P$ ($\mathcal S$ and $P$ finite dimensional) with a Hilbert manifold $M$ to obtain an infinite dimensional singular stratifold in $P \times M$:
$$
f \times id: \mathcal S \times M \to P \times M.
$$

We also define Hilbert stratifolds with boundary. We do it in an analogous
way to the way it is done for finite dimensional stratifolds. A singular 
Hilbert stratifold with boundary in a Hilbert manifold $M$ is given by the following data:\\
1) A topological space $T$ and a closed subspace $\partial T$.\\
2)  A continuous map $f: T \to M$, whose restrictions to $\partial T$ and $T - \partial T$ are singular Hilbert stratifolds.\\
3) A homeomorphism $c:\partial T\times[0,\varepsilon)\to U \subseteq T$, the collar of $T$,  for
some $\varepsilon>0$, where $U$ is an open neighborhood of $\partial T $ in $T$, whose restriction to the boundary is the identity and to the complement of the boundary an isomorphism of stratifolds.\\
4) The map $f$ has to commute with the retract given by the collar $c$, i.e. $fc(x,t) = f(x)$.

There are two things which relate singular Hilbert stratifolds and singular Hilbert
stratifolds with boundary. First, the boundary of a singular Hilbert stratifold
with boundary is itself a singular Hilbert stratifold (without boundary). Secondly, one can 
glue two singular Hilbert stratifolds with boundary along an isomorphism
between their boundary components to obtain a singular Hilbert stratifold. This is done using
the collar of the boundary in the same way as in the
case of finite dimensional stratifolds as described in \cite{K1}. 

A singular Hilbert $k$-stratifold is called orientable if its $-(k+1)$ stratum is empty and the restriction of the map to the $-k$ stratum is orientable in the sense of \cite{KT-o} (that is, an orientation of the determinant line budle). In this case, an orientation is a choice of an orientation of the restriction to the $-k$ stratum. We want to define induced orientations on the boundary.  If $F|_{T - \partial T}$  is a 
singular Hilbert stratifold with boundary with an orientation in the interior of $T$, then there is an induced orientation on the boundary by requiring that the collar $c$ is compatible with the product orientation on the cylinder as explained in \cite{KT-o} and the orientation of $F|_{T - \partial T}$. By construction, an orientation of the cylinder induces opposite orientations on the two ends.

\section{Stratifold Cohomology for Hilbert Manifolds}

With these definitions, all concepts and theorems for the stratifold cohomology of finite dimensional manifolds can be generalized to Hilbert manifolds. 

For the sake of brevity, we give the following definition.
\begin{defn}
A {\bf geometric $k$-cocycle} is an oriented, regular, singular Hilbert k-stratifold, and a {\bf non oriented geometric $k$-cocycle} is a regular, singular Hilbert k-stratifold.  
\end{defn}

\begin{defn}
We define $SH^{k}(M)$ to be the set of geometric $k$-cocycles in $M$ modulo cobordism. Here we also require that the bordisms are oriented, regular singular Hilbert stratifolds. Addition
is given by disjoint union, the inverse is given by reversing the orientation. 
\end{defn}

\begin{rem*}
To see that $SH^{k}(M)$ is a set, note that all Hilbert stratifolds
are of bounded cardinality (continuum) and topologies and sheaves
of real functions defined on a given set is again a set.
\end{rem*}

\subsection{Induced maps}

~

In the finite dimensional setting, for a map $f:M\to N$ which is transversal to a map from a stratifold $g:\mathcal S \to N$ the pull back is a singular stratifold in $M$, which gives a cohomology class in $M$. In general we approximate $f$ by a map which is transversal to $g$ to define the induced map.

The pull back of transversal intersection of Hilbert manifolds has a natural structure of a Hilbert manifold, and similarly for Hilbert stratifolds. It is also easy to see that the pull back of the orientation line bundle is naturally isomorphic to the orientation line bundle of the pull back.  These facts imply that in the case of a submersion the pull back induces a well defined homomorphism in cohomology.
It is also known that every continuous map between Hilbert manifolds is homotopic to a submersion (even an open embedding) by theorem 8.4 in \cite{BK}. The problem is that it is not clear why the induced maps given by two homotopic submersions are equal. Therefore, we give an alternative description for induced maps for general maps.

We start with a definition for submersions, where induced maps are defined using pull back. 
We will reduce the induced map for a general map $f:M\to N$ to the case of submersions by the
standard trick factoring $f$ over $N \times M$. The trick of extending induced maps from submersions to continuous maps is useful in other contexts, hence the details appear in a separate note \cite{KT}. 
To carry this out we need the following proposition, whose proof, which is technical, we postpone to the end of this subsection: 

\begin {prop} \label{iso} 
Let $\pi: E \to M$ be the projection of a vector bundle with fibers finite dimensional or Hilbert space $H$ over a Hilbert manifold $M$. Then 
$$
\pi^*:SH^k(M) \to SH^k(E)
$$
is an isomorphism.
\end{prop}

Assuming that, we define induced maps for the inclusion $i$ of a closed (as topological space)  submanifold $M \subseteq N$ of infinite codimension. To do that, choose a tubular neighborhood $U$ of $M$ in $N$ with projection $\pi$. Given a class in the cohomology of $N$, restrict it to $U$ using the fact that the inclusion of $U$ in $N$ is a submersion. Then apply $(\pi^*)^{-1}:SH^k(U)\to SH^k(M)$ using Proposition \ref{iso}. By the uniqueness theorem for tubular neighborhoods this is independent of the choice of $U$. In general, we define induced maps as follows 
$$
f^* := i^* p^*: SH^k(N) \to SH^k(M),
$$
where $p$ is the projection $M \times N \to N$ and $i$ is the inclusion of $M$ to $M \times N$ given by $(id,f)$. Note, that a homotopy between $f$ and $f'$ induces an isotopy between $(id,f)$ and $(id,f')$ and between their tubular neighbourhoods. This implies that the induced map depends only on the homotopy class of $f$.

This definition is a bit unsatisfactory since it uses the inverse of the map induced by the retraction of the tubular neighbourhood. For concrete computations it is better to consider the map
$$
i^* : SH^k(M) \to SH^k(U)
$$
instead. We will later see, after we have defined the Kronecker product and the linking pairing,  how to decide whether two classes in $SH^k(U)$ agree.

As usual, one can define induced maps for arbitrary continuous maps $f: M \to N$  by approximating $f$ by a homotopic smooth map $g$. 

In the note \cite{KT} we show (in a more general setting) that this definition for a submersion agrees with the original one given by the pull back. In particular, $id^* = id$, and  $(fg)^* = g^*f^*$. Thus we have a contravariant homotopy functor on the category of Hilbert manifolds and continuous maps. 

We would like to relate this definition to the standard definition using transversality. We first prove the following lemma.
\begin{lem}\label{l}
Let $M$ be a Hilbert manifold and $H$ the Hilbert space, and consider the inclusion of the zero section $i:M \to M \times H$. If $\alpha\in SH^k(M\times H)$ is represented by a geometric cocycle $g:\mathcal S \to M\times H$ which is transversal to the zero section then $i^*(\alpha)$ is represented by the pull back $g':\mathcal S' \to M$.
\end{lem}
\begin{proof}
Let $\rho: [0,1] \to [0,1]$ be a smooth function which is $0$ near $0$ and $1$ near $1$. The map $h: M \times H\times [0,1]\to M\times H$ given by $(m,h,t)\mapsto (m,\rho(t) \cdot h)$ is transversal to $g$, since for $\rho (t) \ne 0$ it is a submersion and for $\rho (t)=0$ we use that $g$ is transversal to the zero section. Thus we can consider the pull back of $g$ under $h$ to obtain a cobordism between $g:\mathcal S \to M\times H$ and $\mathcal S'\times H \to M \times H$, which implies the lemma.

\end{proof}
\begin{cor}\label{transversal pull back}
Let $f:M\to N$ be a smooth map between Hilbert manifolds. If $\alpha\in SH^k(N)$ is represented by a geometric cocycle $g:\mathcal S \to N$ which is transversal to $f$ then $f^*(\alpha)$ is represented by the pull back along $f$. 
\end{cor}
\begin{proof}
Let $\pi_N:M\times N\to N$ be the projection and $j:U\to M\times N$ be the inclusion of the tubular neighbourhood of the image of $M$ in $M\times N$ under the embedding $(id,f)$. Then $f^*= (id,f)^*\circ j^* \circ \pi_N^*$. $j$ and $\pi_N$  are submersions, hence $ j^* \circ \pi_N^*$ is represented by the pull back $h:\mathcal T \to U$, which is transversal to $(id,f)$ since  $g:\mathcal S \to N$ is transversal to $f$. The fact that $(id,f)^*$ is given by pull back follows from Lemma \ref{l}. Since the composition of pull backs is the pull back of the composition we are done.
\end{proof}
{\bf The proof of Proposition \ref{iso}} 
\begin{lem} \label{projection induces iso}
Let $M$ be Hilbert manifold which is diffeomorphic to the product $M'\times H$, where $M'$ is a smooth, finite dimensional manifold. Then the projections $M\times H \to M$ and $M\times \mathbb R \to M$ induce isomorphisms in cohomology.
\end{lem}
\begin{proof}
It is enough to prove the statement for $M=M'\times H$. Given an element $[f:\mathcal S\to M'\times H]\in SH^{k}(M'\times H)$, then the composition with the projection  $\pi  \circ f:\mathcal S \to H$ is Fredholm, so by Smale's Theorem \cite{S}, $\pi  \circ f$ has a regular value, say $0 \in H$ (compose with a translation, if necessary), and its preimage is a (finite dimensional) stratifold $\mathcal S'$. Let $\rho: [0,1] \to [0,1]$ be a smooth function which is $0$ near $0$ and $1$ near $1$. The map $h:M'\times H \times [0,1]\to M'\times H$ given by $(m,r,t)\mapsto \big {(} m,\rho(t) \cdot r \big {)}$ is transversal to $f$, since for $\rho (t) \ne 0$ it is a submersion and for $\rho (t)=0$ we use that $0$ is a regular value of $\pi  \circ f$. Thus we can consider the pull back of $f$ under $h$ to obtain a cobordism between $[f]$ and $[\mathcal S'\times H \to M'\times H]$. Using the fact that Hilbert manifolds are stable ($M$ is diffeomorphic to $M\times H$) it is easy to see that this implies that both induced maps are isomorphisms.
 \end{proof}

\begin{cor}\label{coefficients}
For a smooth, oriented, finite dimensional manifold $M'$ the natural transformation $SH^k(M') \to SH^k(M'\times H)$ given by $[\mathcal S'\to M'] \to [\mathcal S'\times H\to M' \times H] $ is an isomorphism. In particular, the coefficients of our theory (which are the cohomology groups of $H$) are give by $$
SH^k(H) = 0 \,\, for \,\, k \ne 0
$$
and
$$
SH^0(H) \cong \mathbb Z.
$$
\end{cor}
\begin{proof}
In Lemma \ref{projection induces iso} we showed that this map is surjective. To show it is injective, just note that the same method will show that if $\mathcal S'\times H\to M'\times H$ bounds, then it is also the boundary of a stratifold of the form $T'\times H\to M'\times H$ with $\partial T'=\mathcal S'$.
\end{proof}

We now define a coboundary operator and prove the exactness of the Mayer-Vietoris sequence in the following setting: Let $M$ be a Hilbert manifold and $U,V$ an open cover, together with a diffeomorphism $\psi:N \times \mathbb R \to U\cap V$, such that the image of the zero section is a closed submanifold of $M$, and $N$ is diffeomorphic to a product of a finite dimensional smooth manifold and $H$. In this case, we define a coboundary operator $\overline{d}:SH^k(U\cap V) \to SH^k(M) $ to be the composition
$$ SH^k(U\cap V) \xrightarrow{\psi^*} SH^k(N \times \mathbb R) \xrightarrow{(\pi^*)^{-1}} SH^k(N) \xrightarrow{i_!} SH^{k+1}(M)$$
where the second map is defined using Lemma \ref{projection induces iso}, and the last map is the Gysin map given by composition $[\mathcal S \to N] \mapsto [\mathcal S \to M]$. Here we use the fact that the inclusion of $N$ in $M$ is proper (its image is a closed submanifold) and Fredholm and oriented (its normal bundle is finite dimensional and oriented).
\begin{lem}
The Mayer-Vietoris sequence in this special case is exact.
\end{lem}
\begin{proof}
The proof of the exactness of the Mayer-Vietoris sequence was worked out for stratifold cohomology in \cite{K1}, page 202ff, and the same arguments apply here. 
\end{proof}

We now prove Proposition \ref{iso}.
\begin{proof} (Proposition \ref{iso})
We start with the case of infinite dimeniosnal fibers. In this case the bundle is trivial. By \cite{EE} Theorem 3E, there exists a sequence of open subsets $Z_n\subseteq M$ ($n\in \mathbb N$) such that the following conditions hold: 1) $Z_n\subseteq Z_{n+1}$ for every $n$. 2) $\cup Z_n =M$. 3) For every $n$, $Z_n$ is diffeomorphic to the product of an $n$-dimensional manifold with the Hilbert space. (They prove it in a more general setting, where $M$ is a Banach manifolds which satisfies certain conditions).

Following Milnor, we look at the telescope $T=\cup_{n=1}^\infty Z_n\times (n,\infty)$ which has the structure of a Hilbert manifold. The projection map to $M$ is a weak homotopy equivalence, hence by Whitehead's theorem together with the fact that Hilbert manifolds have the homotopy type of $CW$ complexes, it is a homotopy equivalence. This implies that $T$ is actually diffeomorphic to  $M$. We decompose $T$ into two open subsets $$U=T \cap \big{(} M \times \cup_{n=0}^\infty (2n,2n+2) \big{)} , \ U=T \cap \big{(} M \times \cup_{n=0}^\infty (2n+1,2n+3) \big{)}.$$
Note that $U$ has the homotopy type (and hence diffeomorphic to) $\cup_{n=0}^\infty  Z_{2n+1}$, $V$ has the homotopy type (and hence diffeomorphic to) $\cup_{n=0}^\infty  Z_{2n+2}$ and $U\cap V$  has the homotopy type (and hence diffeomorphic to) $\cup_{n=0}^\infty  Z_{2n+1}$. Since for every $n$, $Z_n$ is diffeomorphic to the product of the Hilbert space with a finite dimensional manifold, by Lemma \ref{projection induces iso} and additivity, the statement is true for $U,V$ and $U\cap V$. Using the Mayer Vietoris and the five Lemma, the statement is also true for $T$, and hence for $M$. The case where the fibers are finite dimensional follows from the first one by looking at the maps $E\times H \to E \to M$.
\end{proof}

\subsection{A generalized Gysin map}

~

For finite-dimensional closed oriented manifolds $M$ and $N$ of dimension $m$ and $n$, and a map $f: M \to N$ one defines the Gysin map $f_!: H^k(M) \to H^{n-m+k}(N)$ as the conjugation of $f_*$ by the maps given by Poincar\'e duality. If $M$ and $N$ are not closed one can define the Gysin map for proper maps in a similar way using locally finite homology groups. If $M$ and $N$ are not oriented but the map $f$ is oriented one can also define the Gysin map by using twisted locally finite homology groups. 

If $M$ is a smooth closed (as topological space) submanifold in $N$ with oriented normal bundle, then  the Gysin map has a different description as the composition of the Thom isomorphism and the push forward. This description extends to the case of Hilbert manifolds. If $M$ is a Hilbert submanifold of $N$ with $k$-dimensional oriented normal bundle, then one obtains the Gysin map 
$$
i_!: H^l(M;\mathbb Z) \to H^{k+l} (N; \mathbb Z)
$$
by applying the Thom isomorphism and the push forward map. 

This is an important tool. Although simple to define computations are not easy since the Thom class is not constructed but a class which is uniquely characterized  by certain properties and the push forward is constructed using a tubular neighborhood. In our context the Thom class is just the class represented by the zero section and push forward is composition with the inclusion. Thus the Gysin map has a simple description, it is just the composition with the inclusion. This description allows an immediate generalization to proper oriented Fredholm maps $f: M \to N$. 

\begin {defn} For a proper oriented Fredholm map $f: M \to N$ of index $k$ we define the {\bf Gysin map}
$$
f_!: SH^l(M) \to SH^{l-k}(N)
$$
$$
[g: \mathcal S \to M] \mapsto [fg: \mathcal S \to N]
$$
\end{defn}

Here are some properties of the Gysin map. The Gysin map is functorial:

\begin{prop}
If $f:M\to N$ and $g:N \to P$ are proper oriented Fredholm maps, then $g_!\circ f_!=(g\circ f)_!$
\end{prop}
\begin{proof}
This is clear, since in both cases the maps are defined by composition with the composed orientation, which are associative.
\end{proof}

The following proposition relates the Gysin maps and the induced maps:

\begin{prop}\label{commuting with Gysin}
Suppose that the following is a pull back diagram where the maps into $N$ are transversal and the horizontal maps are proper oriented Fredholm maps
\[
\xymatrix{
Q \ar[r]^{\tilde{f}} \ar[d]^{\tilde{g}} & P \ar[d]^{g}\\
M \ar[r]^f & N.}
\]
Then, $$g^*\circ f_! = \tilde{f}_! \circ \tilde{g}^*.$$
\end{prop}
\begin{proof}
We prove it in a few steps.
\begin{enumerate}
\item There are two cases where the statement is clear. The first is when $g$ is a submersion, then it follows from the fact that the composition of pull backs is the pull back of the composition. The second is for diagrams of the form 
\[
\xymatrix{
Q \ar[r]^{\tilde{f}} \ar[d]_{id\times 0} & P \ar[d]^{id\times 0}\\
Q\times V \ar[r]^{\tilde{f}\times id} & P\times V,}
\]
where $V$ is a vector space, simply because every element in $SH^k(Q\times V)$ is represented by a map $S\times V \to Q\times V$. 
\item If the statement holds true for two diagrams, where the bottom map of the first is the top map of the second one then it also holds true for the diagram obtained by concatenating the two.

\item Assume that there is a technical homotopy $h_t$ between $g$ and a submersion $g'$, such that for each $t\in\mathbb R$ $h_t$ is either $g$ or a submersion (in particular transversal to $f$). Let $Q_0=Q$, $Q_1$ the pullback of the diagram which involves $g'$ and $Q_h$ the pullback of the diagram which involves the homotopy, with the inclusions $i_0:Q_0\to Q_h$ and  $i_1:Q_1\to Q_h$:
\[
\xymatrix{
Q_0 \ar[r]^{\tilde{f}} \ar[d]^{\tilde{g}} & P \ar[d]^{g}&Q_1 \ar[r]^{\tilde{f'}} \ar[d]^{\tilde{g'}} & P \ar[d]^{g'}&Q_h \ar[r]^{\tilde{f}_h} \ar[d]^{\tilde{h}} & P\times \mathbb R \ar[d]^{h}\\
M \ar[r]^f & N&M \ar[r]^f & N&M \ar[r]^f & N .}
\]
Since $g\sim g'$ we know that $g'^*=g^*$. The map $g'$ is a submersion, hence $\tilde{f'}_! \circ \tilde{g'}^*=g'^*\circ f_! = g^*\circ f_! $, hence in order to show that $g^*\circ f_! = \tilde{f}_! \circ \tilde{g}^*$  it is enough to show that $\tilde{f}_! \circ \tilde{g}^*=\tilde{f'}_! \circ \tilde{g'}^*$. Note that both $\tilde{g}$ and $\tilde{g'}$ factor through $Q_h$, hence it is enough to show that $\tilde{f}_! \circ i_0^*=\tilde{f'}_! \circ i_1^*$. 
Look at the corresponding diagrams
\[
\xymatrix{
Q_0 \ar[r]^{\tilde{f}} \ar[d]^{i_0} & P \ar[d]^{id\times 0}&Q_1 \ar[r]^{\tilde{f'}} \ar[d]^{i_1} & P \ar[d]^{id\times 1} \\
Q_h \ar[r]^{\tilde{f}_h} & P\times \mathbb R &Q_h \ar[r]^{\tilde{f}_h}  & P\times \mathbb R .}
\]
The vertical maps factor through the tubular neighborhoods, giving us two diagrams such that by (1) the statement holds true, hence by (2) the statement holds true for those diagrams. Since $id\times 0 \sim id \times 1$, their induced maps are equal. Altogether, we get that $\tilde{f}_! \circ i_0^*=\tilde{f'}_! \circ i_1^*$, and the statement holds true for $g$.

\item Consider the following diagram
\[
\xymatrix{
Q \times H \ar[r]^{\tilde{f}\times id} \ar[d]^{\tilde{g}\circ \pi} & P\times H \ar[d]^{g\circ \pi}\\
M \ar[r]^f & N,}
\]
where $H$ is the Hilbert space and $\pi:P\times H \to P$ is the projection. Let $U\subseteq P\times N$ be a tubular neighborhood of $P\xrightarrow{id\times g} P\times N$, and $U\xrightarrow{l} N$ the restriction of the projection. There is a diffeomorphism  $U\cong P \times H$, under which the map $P\to U$ corresponds to the zero section $P\xrightarrow{id\times\{0\}} P\times H$. Let $\rho:\mathbb R \to \mathbb R $ be a smooth, non decreasing map, with $\rho(x)=0$ for all $x\leq 0.1$ and $\rho(x)=1$ for all $x\geq0.9$. Define a homotopy $h:P\times H \times \mathbb R\to N$ by
$$h_t(p,v)=h(p,v,t)=l(p,\rho(t)\cdot v).$$
Note that $h_t$ fulfills the conditions in (3), hence the statement holds true for this diagram.
\item Concatenating the diagram in (1) with the one in (4) we obtain the original diagram. Then, by (2), the statement holds true in general.

\end{enumerate}



\end{proof}

\subsection{Multiplicative structure}

~

Let $M$ and $N$ be two Hilbert manifolds, we define an exterior
product:
\[
SH^{k}(M)\otimes SH^{l}(N)\to SH^{k+l}(M\times N)\]
by setting $[S\to M]\otimes[T\to N]\to(-1)^{kl}[S\times T\to M\times N]$.
We set the orientation of the product using the product orientation (see \cite{KT-o}). With our convention for the product orientation one checks as in \cite{K1},
page 135, that the product is graded commutative. Note, that this agrees with the sign of the product in the finite dimensional case when $\dim M$ is even. If $\dim M$ is odd we cannot use $TM$ as $\eta (\xi)$, since we simplified the signs
occurring there by assuming that $\eta (\xi)$ in the definition of the determinant
line bundle is always even. Nevertheless, in order to prove that the natural transformation is multiplicative it is enough to prove it for the case where $M$ is even dimensional, using the fact that both products are natural and that every odd dimensional manifold $M$ is homotopy equivalent to the even dimensional manifold $M \times \mathbb{R}$.

The cup product is defined in the standard way using pull back along
the diagonal map. Clearly, with this definition, the induced map is
a ring homomorphism. From Corollary \ref{transversal pull back} one can prove the following
\begin{prop}\label{transversal cup product}
If two classes are represented by transversal cocycles then their cup product is represented by their transversal intersection.
\end{prop}

We also have the relation with the Gysin map:

\begin{prop}
For a proper oriented Fredholm map $f: M \to N$ the Gysin map is a map of $SH^*(N)$ modules.
\end{prop}

\begin{proof}
This follows from Proposition \ref{commuting with Gysin} when looking at the following diagram
\[
\xymatrix{
M \ar[d]^{(id,f)} \ar[r]^{f} & N \ar[d]^{(id, id)}\\
M\times N \ar[r]^{(f,id)} & N \times N.}
\]
\end{proof}

\subsection{The cohomology axioms}

~

By a cohomology theory $h$ on the category of Hilbert manifolds and smooth maps we mean a sequence of contravariant homotopy functors $h^k(M)$ together with a natural exact Mayer-Vietoris sequence for open subsets. 

We now define the coboundary operator. This is done by a trick following Dold \cite{Do}. For open subsets $U$ and $V$ of $M$ one replaces the triple $(M;U,V)$ by its thickening $(M';U',V')$. Namely, we replace:\\

\noindent{$M$ by $M':= (U \times  (-\infty,-1)) \cup (U \cap V \times (-\infty,\infty)) \cup (V \times (1 ,\infty))$, }

\noindent{$U$ by $U':= (U \times  (-\infty,-1)) \cup (U \cap V \times (-\infty,\infty))$ and}

\noindent{$V$ by $ V':= (U \cap V \times (-\infty,\infty)) \cup (V \times (1 ,\infty))$. }\\

\noindent Then $U' \cap V' = (U \cap V) \times (-\infty, \infty)$. The advantage is that there is a separating submanifold, namely $(U \cap V) \times \{0\}$. Then the {\bf coboundary operator} is defined as 
$$
d:=    SH^k((U \cap V) \times \{0\})  {\overset{i_!}{\longrightarrow}} SH^{k+1} (U'\cup V') {\overset{(\pi^*)^{-1}}{\longrightarrow}} SH^{k+1}(U \cup V),
$$
where $\pi$ is the projection and $i_!$ is the Gysin map given by composition.  The fact that $\pi^*$ is an isomorphism follows from the fact that $\pi$ is a homotopy equivalence.

\begin{thm}
The Mayer-Vietoris sequence is exact
$$\cdots \to SH^{k-1}(U\cap V) \xrightarrow{d} SH^k(U\cup V) \to SH^k(U)\oplus SH^k(V) \to SH^k(U\cap V)  \to \cdots$$
\end{thm} 
\begin{proof}
This follows from the fact that for the thickened triple the coboundary coincides with the coboundary we had before, and we already know that in this case the Mayer Vietoris sequence is exact.
\end{proof}

We will show that this cohomology theory is isomorphic to singular cohomology by constructing a natural transformation and applying the comparison theorem. Besides the cohomology axioms we need to know the coefficients of the theory, which in the case of a cohomology theory defined only for Hilbert manifolds means to know $SH^*(H)$, the cohomology of the Hilbert space, which was done in Corollary \ref{coefficients}. This completes the proof of  Theorem \ref{thm1}, showing that $SH^k(M)$ is a cohomology theory with Gysin maps. The construction and proof for $SH^k(M;\mathbb Z/2)$ is analogous.

\section{Relation to singular (co)homology}

In this section we construct a natural isomorphism from singular cohomology to stratifold cohomology which commutes with all our structures. Before we do this, we introduce relative cohomology groups which are useful, for example, to formulate the Thom isomorphism Theorem.  We also relate stratifold cohomology to singular homology via a cap product which is used to define the Kronecker product. 

\subsection{Relative cohomology groups and the Thom isomorphism} 

~

In the earlier sections we have defined an absolute cohomology theory, rather than cohomology for pairs. We find this more natural since in the world of manifolds, a priory, only pairs of a manifold and a submanifold make sense. Nevertheless, in certain situations relative groups are useful, for example if we consider the Thom isomorphism. Thus we define relative cohomology groups, but only for a Hilbert manifold and a closed (as topological subspace) submanifold. This was done by Dold for finite dimensional manifolds and singular cobordism \cite{Do}. The same constructions and arguments work in infinite dimensions and singular Hilbert stratifolds. Thus we can be rather  short and refer to Dold for details.

Let $A$ be a closed (as a topological subspace) smooth submanifold of a Hilbert manifold $M$, of finite or infinite dimension. We define the {\bf relative cohomology groups} $SH^k(M,A)$ as the cobordism classes of geometric cocycles $f: \mathcal S \to M \setminus A$, such that $f$  is proper when considered as map to $M$ . 

Now we construct the long exact sequence of the pair $(M,A)$. The restriction map
$$
i^*:SH^k(M,A) \to SH^k(M)
$$
is given by considering for $f: \mathcal S \to M \setminus A$ the map $f$ as a map to $M$. The coboundary operator is defined using a tubular neighborhood $U$ of $A$ in $M$. Denote by $SE$ the sphere bundle with $\pi$ its projection to $A$ and $i$ its inclusion to $M$, then we define $\delta=i_! \circ \pi^*.$ 
This is independent of the choices and gives a well defined homomorphism 
$$
\delta: SH^k(A) \to SH^{k+1} (M,A).
$$
Following the arguments of Dold, one proves that one obtains a {\bf long exact sequence of the pair $(M,A)$}. 

The definition of induced maps is similar to the one in the absolute case, but one has to be a bit more careful.
We first define induced maps for submersions $f: M \to N$ mapping $A \subseteq M$ to $B \subseteq N$. Here $f^*: SH^k(N,B) \to SH^k(M,A)$ is given by pull back. In the special case where $f=p: E \to M$ is a smooth vector bundle, one can look at the sequences of the pairs $(M,A)$ and $(E,E|_A)$, then by the 5-lemma we deduce that the induced map $p^*: SH^r(M,A) \to SH^r(E, E|_A)$ is an isomorphism. Using this we define induced maps for arbitrary smooth maps $f: (M,A) \to (N,B)$ as in the absolute case. To do this we consider a cohomology class represented by a geometric cocycle $g: \mathcal S \to N$. Since $g$ is proper and $N$ is metrizable, $g$ is a closed map \cite{Pal}. Thus there is an open neighborhood $V$ of $B$ which is disjoint from the image of  $g$. Using this one can choose an open tubular neighborhood $\nu \to M \times N$  of $(id,f)(M)$ in $M \times N$, with projection $p$, which is small enough such that the image of $p^{-1}(A)$ under the projection map $\pi:M\times N \to N$ is disjoint from the image of $g$. Therefore, the pull back of $g$ represents a class in $(\nu, p^{-1}(A))$ and we can apply ${p^*}^{-1}$ to get  a class in $SH^k(M,A)$. By standard arguments this is well defined and independent on the choice of the neighborhood $V$. Summarizing we have constructed an {\bf induced map for a map of pairs $f: (M,A) \to (N,B)$}, where $A$ and $B$ are closed (as topological spaces) submanifolds:
$$
f^*: SH^k(N,B) \to SH^k(M,A).
$$

Given induced maps one can define the cup product. If $f: \mathcal S \to M$ represents a class in $SH^k(M,A)$ and $f': \mathcal S' \to M'$ represents a class in $SH^r(M')$, then $f \times f': \mathcal S \times \mathcal S' \to M \times M'$ represents a class in $SH^{k+r}(M \times M',A \times M')$. This induces a product
$$
\times : SH^k(M,A) \otimes SH^r(M') \to SH^{k+r}(M \times M',A \times M').
$$
If  $M = M'$ the composition with the map induced by the diagonal  gives the {\bf relative cup product} 
$$
\cup : SH^k(M, A) \otimes SH^r(M) \to SH^{k+r}(M,A)
$$

Now we consider the Thom isomorphism. Let $p: E \to M$ be a finite dimensional oriented smooth vector bundle equipped with a Riemannian metric. We denote the disc bundle by $DE$ and the sphere bundle by $SE$. We observe that the pair of topological spaces $(DE,SE)$ is homotopy equivalent to the pair $(E,SE)$. Thus we can consider the Thom class of $(E,SE)$ instead of $(DE,SE)$. The Thom class in singular cohomology is not explicitly constructed, one proves that there is a unique class compatible with the local orientation of $E$. In stratifold cohomology the situation is much easier. The $0$-section is a proper oriented Fredholm map and so it represents a cohomology class 
$$
\theta(E) \in SH^k(E,SE),
$$
where $k$ is the dimension of $E$. This is our definition of the {\bf Thom class}. Restricting it to open balls $B$ around a point $x \in M$ one obtains Thom classes of $B \times \mathbb R^n$, which are determined by the local orientation of $E$ at $x$. Thus our Thom class agrees with the Thom class mentioned above. The cup product with the Thom class gives an isomorphism $SH^r(M)  \to SH^{r+k}(E, SE)$. With our definition of the Thom class this map has a very simple explicit interpretation, it maps a geometric cocycle $f: \mathcal S \to M$ to the composition with the $0$-section. And with this interpretation of the cup product the proof of the Thom isomorphism Theorem is an observation:

\begin{thm} (Thom isomorphism) Let $p:E \to M$ be a $k$-dimensional oriented vector bundle. Then the cup product with the Thom class
$$
SH^n(M) \to SH^{n+k}(E,SE)
$$
is equal to the composition
$$
[f: \mathcal S \to M] \mapsto [sf: \mathcal S \to E],
$$
where $s$ is the $0$-section, and it is an isomorphism.

\end {thm}

\begin{proof} Let $\alpha \in SH^k(M)$ be an element represented by a geometric cocycle $g:\mathcal S \to M$. Since $p$ is a submersion, $p^*$ is given by pull back, which is transversal to the zero section, hence by Lemma \ref{transversal cup product} their cup product is given by their intersection, which is given by the composition $\mathcal S \xrightarrow{g} M \xrightarrow{s} E$.

To see it is an isomorphism, note that every class in $SH^{n+k}(E,SE)$ can be represented by a geometric cocycle whose image lies in the open disc bundle  $E_0$. The reason is that every cocycle is the sum of a geometric cocycle whose image lies $E_0$, and a geometric cocycle whose image lies $E_\infty$, the complement of the disc bundle, and the latter is null bordant, since a null bordism can be taken to be the open cylinder pushing the cocycle to $\infty$. For this reason, one can define a Gysin map $SH^{k+n}(E,SE) \to SH^k(M)$, which is the inverse of the Thom map.  
\end {proof}

This gives a description of the Gysin map, in the case of embeddings $f:M \to N$ of a closed (as a topological space) submanifold of finite codimension with oriented normal bundle, in terms of the Thom isomorphism:

\begin{cor}\label{Gysin}
The Gysin map $i_!:SH^n(M)\to SH^{n+k}(N)$ is equal to the composition  $$SH^n(M)\xrightarrow{\pi^*} SH^n(E)\xrightarrow{- \cup\tau}SH^{n+k}(E,SE) \xrightarrow{\phi} SH^{n+k}(N,SE) \to SH^{n+k}(N)$$
where $E$ is the tubular neighbourhood of $M$ in $N$, $SE$ is the sphere bundle and the map $\phi$ is given by composition, using the fact that classes in $SH^{n+k}(E,SE)$ can be represented by maps whose image lies in $E_0$.
\end{cor}

\subsection{A natural isomorphism with singular cohomology}

~

In this subsection we prove a generalization of Theorem \ref{natural iso} to the case of relative cohomology groups. 

\begin{thm} There are natural isomorphisms
$$
\Phi_k: H^k(M,A) \to SH^k(M,A)
$$
commuting with the exterior product, the coboundary in the exact pair sequence and in the absolute case with the coboundary in the Mayer-Vietoris sequence and the Gysin map of an embedding of a closed (as a subspace) submanifold of finite codimension with an oriented normal bundle.
\end{thm}

\begin{proof}
We interpret $H^k(M,A;\mathbb Z)$ as the group of homotopy classes of maps of pairs into the pointed Eilenberg-MacLane space $K(\mathbb Z,k)$ which has a Hilbert manifold model. For each $K(\mathbb Z,k)$ we construct a class $\iota_k\in SH^k(K(\mathbb Z,k),*)$, and $\Phi_k$ will be given by pulling $\iota_k$ back.

For $k=0$, a model for $K(\mathbb Z,0)$ is $\mathbb Z \times H$. We describe $\iota_0$. Let $\mathcal S=\cup_{r \in \mathbb Z} \mathcal S_r$ be the disjoint union of  $\mathcal S_r= N_{|r|} \times H$ where $N_{|r|}$ is a set with $|r|$ elements. The map $S \to \mathbb Z \times H$ is the one induced  by the projections $\mathcal S_r \to \{r\} \times H$. (This means that there are $|r|$ copies of $H$ which are mapped to $\{r\}\times H$). The orientation is the positive one when $r>0$ and negative one when $r<0$. 

For $k>0$,  consider $K(\mathbb Z,k)$  as a Hilbert manifold by starting with $S^k \times H \times [0,1]$ and attaching handles to kill all higher homotopy groups. The interior of this Hilbert manifold is our model for $K(\mathbb Z,k)$. Using the Mayer-Vietoris sequence  the restriction to $S^k \times H \times (0,1)$ induces an isomorphism $SH^k(K(\mathbb Z,k) )\to SH^k(S^k \times H \times (0,1))$. By the Mayer-Vietoris sequence we obtain an isomorphism $SH^k(S^k \times H \times (0,1)) \cong SH^0(H) \cong \mathbb Z$, where a generator is given by the identity map $H \to H$. We denote the corresponding class in $SH^k(K(\mathbb Z,k),*)$ by $\iota_ k$ (we may assume that the image of this geometric cocycle does not contain $*$).

With this we construct a map $\Phi_k :[(M,A),(K(\mathbb Z,k),*)] \to SH^k(M,A)$ by 
$$\Phi_k([f: M \to K(\mathbb Z,k)])=f^*(\iota _k).$$
By construction, $\Phi_k$ commutes with induced maps. To see it is a homomorphism, note that the map induced by $K(\mathbb Z,k) \times K(\mathbb Z,k) \to K(\mathbb Z,k)$ maps $\iota _k$ to $(p_1^*(\iota_k)+ p_2^*(\iota _k))$ by the standard trick.  To show it commutes with the exterior product one only has to show this in the universal case. There it can be reduced to the case of the product of two spheres where the statement is clear.

Let $f:M\to N$ be an embedding of a closed (as a subspace) submanifold of finite codimension with an oriented normal bundle. In Corollary \ref{Gysin} the Gysin map in stratifold cohomology was shown to be equal to composition of induced maps, Thom isomorphism and push forward $SH^{n+k}(E,SE) \xrightarrow{\phi} SH^{n+k}(N,SE)$. A similar description holds for singular cohomology, where the push forward is given by mapping the complement of the disc bundle to $*$. To show that $\Phi$ commutes with the Gysin map, it is enough to prove  that $\Phi$ commutes with  induced maps, with the Thom isomorphism and with the push forward. We already proved the naturality of $\Phi$. The fact that $\Phi$ commutes with the Thom isomorphism follows from the fact that it is multiplicative, and maps the Thom class to the Thom class, because the Thom class in stratifold cohomology satisfies the characterization of the Thom class. Commuting with the push forward is clear by construction.

Since the coboundary operator $d$ in the Mayer-Vietoris can be obtained by the composition of induced maps and Gysin maps for embeddings it follows that $\Phi$ also commutes with $d$.

Thus we have, in the absolute case, a natural transformation of cohomology theories. A similar argument to the one used in Proposition \ref{iso}, using the telescope construction, shows that it is an isomorphism, since this holds true for $X = H$ (Proposition \ref{coefficients}).  Here we use Milnor's additivity axiom \cite{M1} that the cohomology of a disjoint union of manifolds $M_i$ is the direct product of the cohomology groups of the $M_i$'s.

From the fact that $\Phi$ is an isomorphism in the absolute case, the statement for relative cohomology groups follows by the 5-Lemma if $\Phi$ commutes with the coboundary operator in the exact pair sequence. A similar argument to the one used for $d$ in the Mayer-Vietoris shows that $\Phi$ commutes with the coboundary of the pair sequence when those are defined. In stratifold cohomology we defined $\delta$ to be the composition of the  pull back to the sphere bundle and the Gysin map, which is also true in singular cohomology.
\end{proof}

\subsection{The $\cap$-product and the Gysin map in homology, the Kronecker pairing and the linking form} 

~

The cap product plays an important role relating homology and cohomology. We construct it by considering also a geometric version of singular homology using stratifolds. We also describe the linking form, which together with the Kronecker product gives a way of comparing (co)homology classes. The same technique is used later to describe a Gysin map in homology for proper oriented Fredholm maps between Hilbert manifolds. Such a Gysin map in bordism was constructed by Morava \cite{Mo}. He then speculated the existence of such a map in singular homology.  Later Mukherjea \cite{Mu} has constructed a Gysin map in the setting of Fredholm manifolds (Banach manifolds with some extra structure), though he was not able to show that it is independent of the Fredholm filtration.

We start by describing singular homology groups via stratifolds. More precisely we use $p$-stratifolds \cite{K1}, page 23. In \cite {K1} we have used general stratifolds but for the transversality theorem we are proving below we don't see a simple way to get it for general stratifolds. Both sorts of stratifolds describe singular homology of spaces which have the homotopy type of a $CW$-complex, as we will show below. The $p$-stratifolds are stratifolds, which are inductively constructed starting with a zero-dimensional manifold, which is the $0$-stratum, and attaching $k$-dimensional manifolds $W$ to the $(k-1)$-skeleton via smooth proper maps. The sheaf of smooth functions is also constructed inductively after distinguishing a germ of a collar of $\partial W$ in $W$. It is given by the functions, whose restriction to the $(k-1)$-skeleton is contained in the sheaf and whose restriction to the interior of $W$ is smooth and commute with the retract given by a fixed choice of a collar of $\partial W$ in $W$. In addition we assume, like we did it in our description of cohomology, that $\mathcal S$ is a regular stratifold. Bordism classes of closed $k$-dimensional regular   $p$-stratifolds with oriented top stratum and empty $(k-1)$-stratum together with a continuous map to a topological space $X$ form a homology theory $SH_k(X)$. The fact that it is a homology theory was carried out for general stratifolds in \cite{K1} and the same arguments apply for $p$-stratifolds. The relative fundamental class of the top stratum of a closed $p$-stratifold $\mathcal S$ induces a fundamental class $[\mathcal S]\in H_k(\mathcal S;\mathbb Z)$. Assigning to a class $[f: \mathcal S \to X]$ the image of $[\mathcal S]$ gives a natural isomorphism 
\[
\xymatrix{
SH_k(X) \ar[r]^{\cong} & H_k(X;\mathbb Z).
}
\]

Now we give the construction of the {\bf cap product} in a Hilbert manifold $M$

\[
\xymatrix{
\cap : SH^k(M) \otimes SH_l(M) \ar[r]&  SH_{l-k}(M).
}
\]
Let $f : \mathcal S \to M$ be a representative of a class in $SH^k(M)$ and $g : \mathcal P \to M$ a representative of a class in $SH_l(M)$. We say that $g$ is transversal to $f$, if the restriction of both maps to all strata are pairwise transversal, and also each stratum of $\mathcal S$ is transversal to the composition of $g$ with the attaching map $\partial W^k \to P \to M$.  In this case,  the transversal intersection is a compact, oriented $(l-k)$-dimensional stratifold.

\begin{prop}\label{transversality} Let $f: \mathcal S \to M$ be a geometric cocycle and $g: \mathcal P \to M$ be a map from a finite dimensional regular $p$-stratifold $P$ to $M$, which on a closed subset $A$ of $\mathcal P$ is transversal to $f$. Then $g$ is homotopic to $g'$ rel. $A$, such that $g'$ is transversal to $f$.
\end{prop}

\begin{proof}
The key is the transversality theorem by Smale \cite{S} (Theorem 3.1), which says that if $f: N \to M$ is a Fredholm map and $g: V \to M$ is a smooth embedding of a finite dimensional manifold, that $g$ can be approximated by a map transversal to $f$. There is also a relative version, where one assumes $g$ is already transversal to $f$ on a closed subset $A$. We first weaken the assumptions a bit. In the absolute case we don't need to assume that the map $g$ is an embedding since we can approximate all maps from a finite dimensional manifold to a Hilbert manifold by an embedding. In the relative case we restrict ourselves to the case, where $A$ is the boundary of $V$ and $g$ is transversal on the boundary. Now we approximate $g|_{\partial V}$ by an embedding and use a collar to approximate $g$ by a map $g'$, which on $\partial V$ agrees with $g$ and outside an open neighborhood is an embedding. Then we use Smale to approximate $g'$ by a transversal map which on $\partial V$ agrees with $g$. 

We apply this to the case, where $N$ is the disjoint union of all strata of $\mathcal S$, which is itself a Hilbert manifold. Then we proceed inductively over the manifolds $W^n$. Suppose that $g$ restricted to the $n$-skeleton of $\mathcal P$ is already transversal to $f$. Since $\mathcal P$ is regular, this implies that the composition with the attaching map on $\partial W^{n+1}$ is also transversal to $f$. Then we apply our version of Smale's Theorem to approximate $g$ on the $(n+1)$-skeleton by a transversal map. We extend this homotopy inductively to a smooth map on $\mathcal P$ as follows. Suppose that we have extended this homotopy to the $n+r$ skeleton then consider the attaching map of the $n+r+1$ stratum compose it with the already constructed homotopy and use a collar to extend it to the $n+r+1$ skeleton. This is the inductive step for the extension of the homotopy. Using the relative version of Smale's Theorem one can also prove a relative version of this approximation result.
\end{proof}

The {\bf cap product} of a $k$-dimensional cohomology class $[f: \mathcal S \to M]$ and an $l$-dimensional homology class $g: \mathcal P \to M$ is the homology class represented by the transversal intersection which is a $(l-k)$-dimensional oriented compact stratifold. This is well defined by Proposition \ref{transversality}, since if we apply transversality to a  bordism of the (co)homology classes we obtain a $(l-k+1)$-dimensional compact  oriented stratifold whose boundary is the $(l-k)$-dimensional transversal intersection.

The {\bf Kronecker product} $$<..., ...>: SH^k(M) \times SH_k(M) \to \mathbb Z$$ is defined by the cap product followed by the map $SH_0(M)\to SH_0(pt) \cong \mathbb Z$. This agrees with the ordinary Kronecker product in singular (co)homology under our isomorphism. The reason is that this holds by construction in the universal case, where we evaluate $i_k \in SH^k(K(\mathbb Z,k))$ on the generator in homology given by the inclusion of $S^k$ into $K (\mathbb Z,k)$. 

We also have a linking pairing between a torsion class $\alpha \in SH^k(M)$ and a torsion class $\beta$ in $SH_{k-1}(M)$, which is obtained by taking a zero bordism of a multiple  $r\alpha$, taking the Kronecker product with $\beta$, dividing by $r$ and considering the result in $\mathbb Q/_\mathbb Z$. Thus we obtain the {\bf linking pairing}
$$
l: tor \,\, SH^k(M) \otimes tor \,\, SH_{k-1}(M) \to \mathbb Q/_\mathbb Z.
$$
The linking pairing together with the Kronecker product allows (by the Universal Coefficient Theorem) to compare cohomology classes. Let $\alpha$ and $\beta$ be classes in $SH^k(M)$ where $M$ is a Hilbert manifold, then  $\alpha=\beta$ if and only if for all homology classes $x\in SH_k(M)$ the Kronecker product $<\alpha,x> = <\beta,x>$ agrees, which implies that $\alpha - \beta $ is a torsion class, and for all torsion classes $y \in SH_{k-1}(M)$ we have $l(\alpha - \beta,y) =0$. This have the following application: 

\begin{prop} \label{13}Let $f: M \to N$ be a homotopy equivalence. Then for $\alpha$ and $\beta$ in $SH^k(N)$ we have
$$
(f^*)^{-1} (\alpha) = (f^*)^{-1} (\beta), 
$$
if and only if for all all homology classes $x\in SH_k(M)$ the Kronecker product $<\alpha,f_*(x)> = <\beta,f_*(x)>$ agree,  and for all torsion classes $y \in SH_{k-1}(M)$ we have $l(\alpha - \beta,f_*(y)) =0$.
\end{prop}

\begin{rem}
This proposition is useful, when one computes induced maps $f^*: SH^k(N) \to SH^k(M)$, and instead considers $i^*(p_2)^*: SH^k(N) \to SH^k(U)$, where $U$ is a tubular neighborhood of $(id,f)(M) \subseteq M \times N$. 
\end{rem}

For a proper oriented Fredholm map $f:M\to N$ of index $k$ we define the {\bf Gysin map} in homology 
$$f^!:H_l(N) \to H_{l+k}(M)$$
using transversal intersection. To be more precise, we represent a class $\alpha\in H_l(N)$ by a map $g:P\to N$ from a closed oriented $p$-stratifold $S$ of dimension $l+k$.  We may assume that $g$ is transversal to $f$, then we look at the pull back, which is map from a closed oriented $p$-stratifold $P'$ of dimension $l+k$ to $M$. As before, this is independent of the choice of $P$ and $g$, thus it is well defined.

\section{equivariant cohomology}
As noted before, one of our motivations for this construction was a geometric description of equivariant cohomology for smooth actions of compact Lie groups on finite dimensional smooth manifolds and on Hilbert manifolds. We now describe it explicitly:
Let $M$ be a Hilbert manifold (or finite dimensional smooth manifold) with a smooth action of a compact Lie group $G$. We construct a Hilbert manifold model for $EG$, a contractible Hilbert manifold with free $G$-action as follows. Since $G$ can be embedded into $O(n)$ for some $n$, it is enough to construct $E(O(n))$. This is more or less the same as the standard construction of $E(O(n))$ as limit of Stiefel manifolds, or even easier. Namely, we can take the Stiefel manifold $V_n(H)$ of orthonormal $n$-frames in the Hilbert space $H$ (see \cite{E}, p. 53). Such a model is unique up to equivariant diffeomorphism. Given this,  we consider geometric cocycles with an orientation preserving, free $G$-action in $M\times EG$, i.e. a geometric cocycle  $f:\mathcal S \to M\times EG$, with a free, smooth action of $G$ on $S$ such that $f$ is equivariant and proper, and $G$ acts orientation preserving on the orientation bundle $L(f)$. 
We denote the bordism classes of equivariant, oriented, free,  $G$-geometric $k$-cocycle $f: S \to M \times EG$, where $G$ acts orientation preserving on $L(f)$, by
$$
SH^{k}_G(M),
$$
the {\bf $k^{th}$ equivariant stratifold
cohomology group} of $M$.  The addition is given by disjoint union. 
\begin{thm}  Let $G$ be a compact Lie group. 
There is  a natural isomorphism of multiplicative
equivariant cohomology theories  
$$
SH_{G}^{*}(M) \to H_{G}^{*}(M;\mathbb{Z})
$$ 
for Hilbert (or finite dimensional) smooth $G$-manifolds $M$.

\end{thm}

\begin{proof}
This is straight forward by passing to the quotients, and then using Theorem \ref{thm1}. 
\end{proof}

In \cite{T2} more general theories along these lines were studied.

\section{Degree}

A notion of degree of proper Fredholm maps of index zero between Banach manifolds was given by Smale \cite{S}. His degree was an invariant in $\mathbb Z /2$. A more refined notion of degree  for oriented maps laying in $\mathbb Z$ was given later by Elworthy-Tromba \cite{ET} and others. In our  language, a proper, oriented Fredholm map of index zero $f:M\to N$ between Hilbert manifolds can be seen as a  cocycle representing a cohomology class in $SH^0(N)$. Assuming $N$ is connected this group is naturally isomorphic to $\mathbb Z$, hence we can associate to such a map a degree, which we denote by $deg(f)$. In this sense, associating to a proper oriented Fredholm map $f:M\to N$ of index $k$  an element in $SH^k(N)$ is a generalization of the notion of degree. This notion agrees with Morava's definition \cite{Mo}. We apply this definition to extend some well known interesting results to the setting of Hilbert manifolds and proper oriented Fredholm maps.

\begin{lem}\label{injectivity}
Let $f:M\to N$ be a proper oriented Fredholm map of index $k$. Then for every $\alpha \in SH^l(N)$
$$f_! \circ f^* (\alpha)= [f] \cup \alpha$$ 
\end{lem}
\begin{proof}
The following is a pull back diagram where the maps into $N\times N$ are transversal
\[
\xymatrix{
M \ar[r]^f \ar[d]^{(f,id)} & N \ar[d]^{(id,id)}\\
N\times M \ar[r]^{(id,f)} & N\times N.}
\]
This implies, by Proposition \ref{commuting with Gysin}, that 
$$(id,id)^*\circ (id,f)_! = f_! \circ (f,id)^*.$$
Represent $\alpha$ by a map $g:S\to N$. Then, $f_! \circ f^* (\alpha)$ is equal to $(id,id)^*([S\times M \to M\times M])=\alpha \cup [f]$.
\end{proof}

The following consequences are standard:

\begin{cor}
If $(- \cup [f]):SH^*(N) \to SH^*(N)$ is injective then $f^*$ is injective. In particular, if $k=0$ and $deg(f)$ is non zero then $f^*$ is injective modulo $deg(f)$-torsion.
\end{cor}

\begin{example} 
Let $f:M\to N$ be a proper oriented Fredholm map. If the cohomology ring of $M$ is an integral domain, for example a polynomial ring, and the cohomology of $N$ does not contain a subring isomorphic to the cohomology of $M$ then $f$ represents the zero class.
\end{example}

\begin{prop} (compare \cite{Mu})
Let $f:M\to N$ be proper oriented Fredholm map of index $0$ between connected manifolds. If $deg(f)=\pm 1$ then if $M$ is contractible then also $N$ is contractible and so diffeomorphic to the Hilbert space $H$.
\end{prop}
\begin{proof}
By Lemma \ref{injectivity} we know that $N$ is acyclic. Since $N$ has the homotopy type of a $CW$ complex it is enough to show it is simply connected.   Then we consider a map $g:S^1\to N$ representing an element of $\pi_1(N)$. By the transversality Theorem for maps from a finite dimensional manifold we can assume that $g$ is transversal to $f$. We can further approximate it by an embedding which is transversal to $f$. Then $S:= f^{-1}(g(S^1))$ is a closed oriented submanifold of dimension $1$, since $f$ has index $0$ and $f|_S$ is a submersion. Thus $f|_S$ is a covering on each connected component $S_i$ of $S$. The sum of the degrees $d_i$ of these coverings $S_i\to g(S^1) $ is  the degree of $f$. If we restrict $f$ to one of the components $S_i$ and use the fact that $M$ is simply connected we conclude that $d_i[f] =0$. Since $\sum d_i =1$ we conclude that $[f] =0$.
\end{proof}

\section{Characteristic classes and Steenrod representation}\label{char}

Characteristic classes of n-dimensional smooth vector bundles $p: E \to M$ over a Hilbert manifold $M$ can be geometrically defined as in \cite{K1}. The Thom class of an oriented vector bundle, which is a class in $SH^n(DE, SE)$, where $DE$ is the disc bundle and $SE$ the sphere bundle, is simply given by the $0$ section. To see this is indeed the Thom class, just notice that it intersects each fiber in exactly one point with the right orientation. This extends to an absolute  class in the total space, whose image under the map induced by the $0$-section is the {\bf Euler class}. For non-oriented bundles the same construction gives the $n$-th {\bf Stiefel-Whitney class}.

From these classes one constructs the Chern classes of a complex vector bundle by considering the exterior tensor product of $E$ with the tautological line bundle over $\mathbb {C}P^\infty$ considered as Hilbert manifold. One considers the Euler class of this bundle to obtain the {\bf total Chern class}
$$
c(E) \in SH^{2n} (M \times \mathbb {C}P^\infty).
$$
If one applies the K\"unneth Theorem to this class one obtains as coefficients the Chern classes $c_k(E) \in SH^{2k}(M)$. By naturality the classes $c_k(E)$ can be directly constructed as follows.
Instead of taking the Euler class of the exterior tensor product with the tautological line bundle over $\mathbb {C}P^\infty$ we take the 
Euler class of the exterior tensor product with the tautological line bundle over $\mathbb {C}P^{n-k}$. This is a class in $SH^{2n} (M \times \mathbb {C}P^{n-k})$ and we obtain the Chern class $c_{k}(E)$ by composing it with the projection to $M$ (with other words we apply the Gysin map of the projection to $M$). From the Chern classes one defines the Pontryagin classes. The Stiefel-Whitney classes are constructed in the same way from the top Stiefel-Whitney class by using the real projective spaces instead of the complex projective spaces. 

As in the finite dimensional setting the proof that these classes agree with the ordinary characteristic classes follows by showing that the characteristic classes fulfill the axioms of characteristic classes. The proof is the same as in \cite {K1}. 

If one wants to use this definition for explicit computations one has a problem already with the Euler class. The Thom class in $SH^k(DE,SE)$ is explicit, since it is the class represented by the $0$-section, but the Euler class is the pull back, which in the infinite dimensional setting is not given by the transversal self intersection unless one cam deform the inclusion of the $0$-section to make it transversal. Thus it is better to consider the Euler class sitting in $SH^k(E)$ represented by the $0$-section. Similarly one should consider the total Chern class as sitting in $SH^k(E \otimes L)$, where $E \otimes L$ is the exterior tensor product of $E$ and the tautological bundle over $\mathbb {C}P ^N$ for a large $N$. 

This raises the question whether keeping the classes in the cohomology of the total spaces is a disadvantage. Not really.  We recommend for computations to consider the characteristic classes sitting in the cohomology of $E$, in the case of the Euler class, and in the cohomology of $E  \otimes L$ in the case of the total Chern class. Using the Kronecker product and the linking pairing one can then compare this class with cohomology classes in $M$ resp.\ $M \times \mathbb {C}P^N$ using Proposition \ref{13}.

For finite dimensional oriented manifolds one studies the Steenrod representation problem for (co)homology classes. This means for a cohomology class $x \in H^k(M)$ to ask whether there is a  map $f: N \to M$ from an oriented manifold such that $f_! (1) = \alpha$, where $1$ means the unit class in $H^0(N)$. Using our definition of singular cohomology one can generalize this in two ways: Instead of looking for a closed oriented manifold one can look for a proper oriented map $f: N \to M$, such that the cohomology class represented by $f$ in $SH^k(M)$ corresponds to $\alpha$. And verbally the same makes sense in infinite dimensions, where $M$ is a Hilbert manifold and we ask for a proper oriented Fredholm map representing $\alpha$. This is the {\bf Steenrod representation problem for Hilbert manifolds}. 

Since the Euler class is represented by a manifold (even by a submanifold) we see that the Chern classes (as well as the Stiefel-Whitney classes) are Steenrod representable in the sense that there is a proper oriented Fredholm map from a smooth Hilbert manifold to $M$ representing the class in $SH^{2k}(M)$. Since the Pontryagin classes are derived from the Chern classes the same holds for them. 

Thus we  obtain the following result:

 \begin{thm} Let $E$ be a $k$-dimensional vector bundle over a Hilbert manifold $M$. Then all characteristic classes (Stiefel-Whitney classes, Euler class, if $E$ is oriented, Chern classes, if $E$ is a complex vector bundle, and Pontryagin classes) are in the sense above Steenrod representable.
 \end{thm}
 
 \begin {rem} Probably one can prove a better result. Consider a Hilbert space model for $BO$ or $BU$ (see for example \cite{PS}). Then one can ask for the existence of Hilbert submanifolds representing the Stiefel-Whitney classes, the Pontryagin classes or the Chern classes. In the context of Banach manifold models this was studied by Koschorke \cite{Ko}. 
 \end{rem}

\end{document}